\documentclass[11pt]{amsart}

\usepackage{amsmath,amssymb,amsthm}
\usepackage{color, tikz}
\usepackage{tabularx}

\usepackage[margin=1in]{geometry}

\newcolumntype{Y}{>{\centering\arraybackslash}X}

\newtheorem{theorem}{Theorem}[section]

\newtheorem{claim}[theorem]{Claim}

\newtheorem{corollary}[theorem]{Corollary}
\newtheorem{algorithm}[theorem]{Algorithm}
\theoremstyle{definition}
\newtheorem{remark}[theorem]{Remark}
\newtheorem{definition}[theorem]{Definition}

\newtheorem{example}[theorem]{Example}

\newcommand{\ZZ}{\mathbb{Z}}

\newcommand{\Hom}{\mathrm{Hom}}
\newcommand{\F}{\mathcal{F}}
\newcommand{\bfs}[1]{\mbox{\boldmath$#1$}}
\newcommand{\bi}{\mathbf{i}}
\newcommand{\Des}{\mathrm{Des}}
\newcommand{\veps}{\varepsilon}

\newcommand{\todo}[1]{\par \noindent
  \framebox{\begin{minipage}[c]{0.95 \textwidth} TO DO:
      #1 \end{minipage}}\par}

\newcommand\commentout[1]{}

\begin{document}

\title{Homomorphism Complexes and Maximal Chains in Graded Posets}

\author{Benjamin Braun}
\address{Department of Mathematics\\
  University of Kentucky\\
  Lexington, KY 40506--0027}
\email{benjamin.braun@uky.edu}

\author{Wesley K. Hough}
\address{Department of Mathematics\\
  University of Wisconsin - Whitewater\\
  Whitewater, WI 53190 U.S.A.}
\email{houghw@uww.edu}

\date{18 December 2018}

\thanks{
  This research was partially supported by a grant from the College of Letters and Sciences at the University of Wisconsin - Whitewater.
  The authors thank Eric Kaper for suggestions regarding Theorem~\ref{thm:eulerrecur} and Bert Guillou for helpful conversations.
}

\begin{abstract}
  We apply the homomorphism complex construction to partially ordered sets, introducing a new topological construction based on the set of maximal chains in a graded poset.
  Our primary objects of study are distributive lattices, with special emphasis on finite products of chains.
  For the special case of a Boolean algebra, we observe that the corresponding homomorphism complex is isomorphic to the subcomplex of cubical cells in a permutahedron.
  Thus, this work can be interpreted as a generalization of the study of these complexes.
  We provide a detailed investigation when our poset is a product of chains, in which case we find an optimal discrete Morse matching and prove that the corresponding complex is torsion-free.
\end{abstract}

\maketitle

\section{Introduction}

Homomorphism complexes were introduced by Lov\'asz as a generalization of the neighborhood complex of a finite simple graph. 
Babson and Kozlov~\cite{BabsonKozlovComplexes,BabsonKozlovLovaszConjecture} proved that graph homomorphism complexes did indeed produce new and interesting topological spaces that led to topological lower bounds for graph chromatic numbers.
However, Schultz~\cite{SchultzSpacesOfCircuits} proved that these lower bounds were generally not better than the original topological lower bounds obtained by Lov\'asz~\cite{LovaszChromaticNumberHomotopy} using the neighborhood complex.
While graph homomorphism complexes have not yet provided the powerful generalization that was originally hoped for, they have generated interesting mathematical developments; further, the general homomorphism complex construction has found other interesting mathematical applications, for example in PL-manifold theory~\cite{csorbalutz,schultzmanifolds}, topological properties of set systems~\cite{jonssonhomset}, and cellular resolutions of monomial ideals~\cite{BraunBrowderKlee,DochtermannEngstromCellular}.

In this work, we apply the homomorphism complex construction to partially ordered sets, introducing a new topological construction $\Hom(C_m,P)$ based on the set of maximal chains in a graded poset $P$.
Our primary object of study is $\Hom(C_m,P)$ when $P$ is a distributive lattice, and we give a detailed analysis of those lattices formed as a finite product of chains.
When $P$ is a Boolean algebra, we observe that $\Hom(C_m,P)$ is isomorphic to the subcomplex of cubical cells in a permutahedron; thus, this work can be interpreted as a generalization of the study of these complexes.
One such generalization has already been undertaken by Severs and White~\cite{severswhite}, and in the case where $P$ is a Boolean algebra, our proof techniques align with theirs.
Our main tool is discrete Morse theory, which we use to produce optimal discrete Morse functions on $\Hom(C_m,P)$ when $P$ is a product of chains; we find that these complexes are torsion-free in all homological dimensions.

The remainder of our paper is structured as follows.
Section~\ref{sec:discretemorsetheory} contains necessary background regarding discrete Morse theory.
Section~\ref{sec:homcomplexes} introduces the general homomorphism complex construction, along with the specific construction for posets we will consider.
Section~\ref{sec:maxchains} introduces the complex $\Hom(C_m,P)$ of maximal chains in a graded poset and establishes basic properties when $P$ is a distributive lattice.
Section~\ref{sec:productsofchains} contains our main result, an optimal discrete Morse function on $\Hom(C_m,P)$ when $P$ is a product of chains, as well as a proof that the homology groups of $\Hom(C_m,P)$ are torsion-free.
Finally, in Section~\ref{sec:examples} we provide examples and corollaries to our main theorems regarding products of chains.

\section{Discrete Morse Theory}\label{sec:discretemorsetheory}

Discrete Morse theory was first developed by R. Forman in \cite{FormanMorseTheory} and has since become a powerful tool for topological combinatorialists.
The main idea of the theory is to systematically pair off faces within a polyhedral cell complex in such a way that we obtain a collapsing order for the complex, yielding a homotopy equivalent cell complex.

\begin{definition}

  A \textit{partial matching} in a poset $P$ is a partial matching in the underlying graph of the Hasse diagram
  of $P$, i.e., it is a subset $M\subseteq P \times P$ such that
  \begin{itemize}
  \item
    $(a,b)\in M$ implies $b \succ a;$ i.e. $a<b$ and no $c$ satisifies $a<c<b$.
  \item
    each $a\in P$ belongs to at most one element in $M$.
  \end{itemize}
  When $(a,b) \in M$, we write $a=d(b)$ and $b=u(a)$.
\item
  A partial matching on $P$ is called \emph{acyclic} if there does not exist a cycle
  \[
    a_1 \prec u(a_1) \succ a_2 \prec u(a_2) \succ \cdots \prec u(a_m) \succ a_1
  \]
  with $m\ge 2$ and all $a_i\in P$ being distinct.

\end{definition}

Given an acyclic partial matching $M$ on a poset $P$, we call an element $c$ \emph{critical} if it is unmatched.
If every element is matched by $M$, we say $M$ is \emph{perfect}.
We are now able to state the main theorem of discrete Morse theory as given in~\cite[Theorem 11.13]{KozlovBook}

\begin{theorem}\label{thm:dmt}
  Let $\Delta$ be a polyhedral cell complex, and let $M$ be an acyclic matching on the face poset of $\Delta$.
  Let $c_i$ denote the number of critical $i$-dimensional cells of $\Delta$.
  \begin{enumerate}
  \item[(a)] The space $\Delta$ is homotopy equivalent to a cell complex $\Delta_c$, called the \emph{Morse complex of $\Delta$ with respect to $M$}, with $c_i$ cells of dimension $i$ for each $i\ge 0$, plus a single $0$-dimensional cell in the case where the empty set is paired in the matching.
  \item[(b)] There is an indexing of the cells of $\Delta_c$ with the critical cells of $\Delta$ such that for any two cells $\sigma$ and $\tau$ of $\Delta_c$ satisfying $\dim(\sigma)=\dim(\tau)+1$, the incidence number $[\tau:\sigma]$ in the cellular chain complex for $\Delta_c$ is given by
    \[
      [\tau:\sigma]=\sum_cw(c)
    \]
    where the sum is taken over all alternating paths $c$ connecting $\sigma$ with $\tau$, i.e. with all sequences $c=(\sigma,a_1,u(a_1),\ldots,a_t,u(a_t),\tau)$ such that $a_1 \prec \sigma$, $\tau\prec u(a_t)$, and $a_{i+1}\prec u(a_i)$ for $i=1,\ldots,t-1$.
    For such an alternating path, the quantity $w(c)$ is defined by
    \begin{equation}\label{eqn:weight}
      w(c):=(-1)^t[a_1:\sigma][\tau:u(a_t)]\prod_{i=1}^t[a_i:u(a_i)]\prod_{i=1}^{t-1}[a_{i+1}:u(a_i)] \, ,
    \end{equation}
    where the incidence numbers on the right-hand side are taken in the complex $\Delta$.
  \end{enumerate}
\end{theorem}

It is often useful to create acyclic partial matchings on several different sections of the face poset of a polyhedral cell complex and then combine
them to form a larger acyclic partial matching on the entire poset.
This process is detailed in the following theorem known as the \textit{Cluster Lemma} in \cite{JonssonBook} and the \textit{Patchwork Theorem} in \cite{KozlovBook}.

\begin{theorem}\label{patchwork}
  Assume that $\varphi : P \rightarrow Q$ is an order-preserving map.
  For any collection of acyclic matchings on the subposets $\varphi^{-1}(q)$ for $q\in Q$, the union of these matchings is itself an acyclic matching on $P$.
\end{theorem}


\section{Homomorphism Complexes}\label{sec:homcomplexes}

\subsection{General Constructions}
Suppose that one has two finite sets $A$ and $B$ and a collection $M$ of maps from $A$ to $B$ that we refer to as homomorphisms.
We can endow $M$ with a topological structure as follows.
This definition was originated by Kozlov~\cite[Definition 9.24]{KozlovBook}, who suggested applying it to a broad range of objects including the posets we consider in this work.

\begin{definition}
  Let $A$ and $B$ be finite sets.  
  Let $M$  be a collection of maps from $A$ to $B$ called homomorphisms.  
  Let $P(A,B)$ be the polyhedral cell complex $\prod_{i\in A}\Delta_B$, where $\Delta_B$ is the simplex with vertices labeled by elements of $B$.  
  We record a face of $P(A,B)$ by an $|A|$-tuple of the form $X=(X_i)_{i\in A}$, with each $\emptyset \neq X_i\subseteq B$.
  The \emph{homomorphism complex} $\Hom_M(A,B)$ is the subcomplex of $P(A,B)$ consisting of all cells labeled by $(W_i)_{i\in A}\in P(A,B)$ satisfying the following property: if $\eta:A\rightarrow B$ is such that $\eta(i)\in W_i$ for all $i\in A$, then $\eta\in M$.  
  We call such a $(W_i)_{i\in A}\in P(A,B)$ a \textit{multi-homomorphism}, and we call $A$ the \textit{test object} and $B$ the \textit{target object} for the complex.
  The faces of $\Hom_M(A,B)$ are partially ordered by $X\leq Y$ if and only if $X_i\subseteq Y_i$ for all $i\in A$.  
\end{definition}

Given a cell $(W_i)_{i\in A}\in\Hom_M(A,B)$, we will usually write $W_i=\eta(i)$, where $\eta$ is understood to be a set-valued function.
The cells in $\Hom_M(A,B)$ are all products of simplices; such complexes are called \emph{prodsimplicial} and are special examples of polyhedral cell complexes.
For two graphs $H$ and $G$, we define the \emph{graph homomorphism complex} $\Hom(H,G)$ to be the homomorphism complex resulting from defining $M$ to be the set of edge-preserving maps from $V(H)$ to $V(G)$.
This was the original context in which homomorphism complexes were developed, though the construction is much more general.

When $A$ and $B$ are finite posets, then there are many potentially-interesting choices of maps $M$ to use in constructing a homomorphism complex.
In this paper, we will use the following definition.

\begin{definition}\label{def:posetstricthom}
  Let $P$ and $Q$ be finite posets, and let $M$ be the set of strictly-order-preserving maps from $P$ to $Q$.
  The \emph{poset homomorphism complex} $\Hom(P,Q)$ is defined to be $\Hom_M(P,Q)$ for this choice of $M$.
\end{definition}

\begin{example}\label{ex:chainbooleanhom}
  Figure~\ref{fig:chainboolean} shows a homomorphism from $C_3=\{0<1<2<3\}$ into a Boolean algebra on $3$ elements.
  Figure~\ref{fig:chainbooleanhom} shows the corresponding homomorphism complex, with the vertices corresponding to homomorphisms labeled $(\eta(0),\eta(1),\eta(2),\eta(3))$.
  The edge connecting $(\emptyset,1,12,123)$ with $(\emptyset,2,12,123)$ corresponds to the multihomomorphism $(\emptyset,\{1,2\},12,123)$.
\end{example}

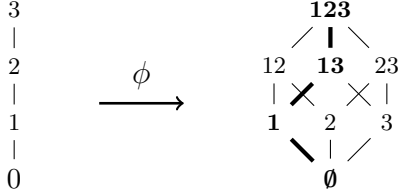
\begin{figure}
  \centering
  \begin{tikzpicture}[scale=0.75]
    \node (0) at (0,0) {$0$};
    \node (1) at (0,1) {\footnotesize $1$};
    \node (2) at (0,2) {\footnotesize $2$};
    \node (3) at (0,3) {\footnotesize $3$};
    \draw (0) -- (1) -- (2) -- (3);
  \end{tikzpicture}
  \hspace{0.25in}
  \begin{tikzpicture}[scale=0.75]
    \draw [line width = 1pt, ->] (0,0)--(1.5,0);
    \draw node at (0.75,0.5) {\large$\phi$};
    \draw node at (0.75,-1.5) {};
  \end{tikzpicture}
  \hspace{0.25in}
  \begin{tikzpicture}[scale=0.75]
    \node [] (empty) at (1,0) {\footnotesize $\bfs{\emptyset}$};
    \node [] (1) at (0,1) {\footnotesize \textbf{1}};
    \node (2) at (1,1) {\footnotesize $2$};
    \node (3) at (2,1) {\footnotesize $3$};
    \node (12) at (0,2) {\footnotesize $12$};
    \node [] (13) at (1,2) {\footnotesize \textbf{13}};
    \node (23) at (2,2) {\footnotesize $23$};
    \node [] (123) at (1,3) {\footnotesize \textbf{123}};
    \draw [ultra thick] (empty)--(1);
    \draw (empty)--(2);
    \draw (empty)--(3);
    \draw (1)--(12);
    \draw [ultra thick] (1)--(13);
    \draw (2)--(12);
    \draw (2)--(23);
    \draw (3)--(13);
    \draw (3)--(23);
    \draw (12)--(123);
    \draw [ultra thick] (13)--(123);
    \draw (23)--(123);
  \end{tikzpicture}
  \caption{A homomorphism from $C_3$ into a Boolean algebra on $3$ elements.}
  \label{fig:chainboolean}
\end{figure}

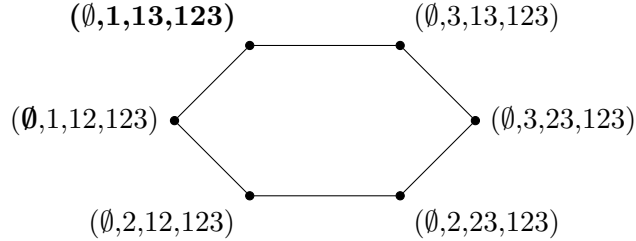
\begin{figure}
  \centering
  \begin{tikzpicture}
    \node[draw,circle,black,fill=black,scale=0.3,label=left:{($\bfs{\emptyset}$,1,12,123)}] (1) at (-2,0) {};
    \node[draw,circle,black,fill=black,scale=0.3,label=above left:{\textbf{($\emptyset$,1,13,123)}}] (2) at (-1,1){};
    \node[draw,circle,black,fill=black,scale=0.3,label=above right:{($\emptyset$,3,13,123)}] (3) at (1,1){}; 
    \node[draw,circle,black,fill=black,scale=0.3,label=right:{($\emptyset$,3,23,123)}] (4) at (2,0){}; 
    \node[draw,circle,black,fill=black,scale=0.3,label=below right:{($\emptyset$,2,23,123)}] (5) at (1,-1){};
    \node[draw,circle,black,fill=black,scale=0.3,label=below left:{($\emptyset$,2,12,123)}] (6) at (-1,-1){};
    \draw (1)--(2)--(3)--(4)--(5)--(6)--(1);
    
  \end{tikzpicture}
  \caption{The homomorphism complex from $C_3$ into a Boolean algebra on $3$ elements, with the homomorphism from Figure~\ref{fig:chainboolean} in bold.}
  \label{fig:chainbooleanhom}
\end{figure}


\subsection{Foldings}

In the graph context, the homotopy type of $\Hom(H,G)$ is preserved by a ``folding'' operation on either $H$ or $G$, defined as follows.

\begin{definition}
  Let $\mathcal{N}(v)$ denote the neighborhood of $v$ in $G$.
  We say $G - v$ is a \textbf{fold} of $G$ if there exists $u \in V(G)$ with $u \neq v$ such that $\mathcal{N}(u) \supseteq \mathcal{N}(v)$.
\end{definition}

For a polyhedral cell complex $X$, let $Bd \ X$ denote the barycentric subdivision of $X$.

\begin{theorem} (Kozlov \cite{KozlovSimpleFold})\label{thm:graphfold}
  Let $G - v$ be a fold of $G$, and let $H$ be some graph.
  Then, $Bd \ \Hom(G,H)$ collapses onto $Bd \ \Hom(G-v,H)$, and $\Hom(H,G)$ collapses onto $\Hom(H,G-v)$.
\end{theorem}

Therefore, there exists a suitable neighborhood condition that allows us to effectively ignore certain vertices in either the test or target graph.
We next make an analogous definition of folding for posets.

\begin{definition}
  Let $P$ be a finite poset with $x\in P$.
  Let $\mathcal{U}(x)$ denote the set of elements that cover $x$, and let $\mathcal{D}(x)$ denote the set of elements covered by $x$.
  We say that $P - x$ is a \textbf{fold} of $P$ if there exists $y \in P$ with $x \neq y$ such that $\mathcal{U}(y) \supseteq \mathcal{U}(x)$ and $\mathcal{D}(y) \supseteq \mathcal{D}(x)$.
\end{definition}

Observe that $\mathcal{U}(x) \cup \mathcal{D}(x)$ is equal to the graph-theoretic neighborhood of $x$ in the Hasse diagram for $P$.
This modified notion of folds and poset neighborhoods yields the following.

\begin{theorem} \label{posetfold}
  Let $P - x$ be a fold of $P$ with some poset $Q$.
  Then, $Bd \ \Hom(P,Q)$ collapses onto $Bd \ \Hom(P-x,Q)$, whereas $\Hom(Q,P)$ collapses onto $\Hom(Q,P-x)$.
\end{theorem}

To prove this, we need only slightly modify Kozlov's original proof in~\cite{KozlovSimpleFold} of Theorem~\ref{thm:graphfold}.
We first review basic facts about closure operators.

\begin{definition}
  An order-preserving map $\phi$ from a poset $P$ to itself is a \textit{descending closure operator} if $\phi^2 = \phi$ and $\phi(x) \leq x$ for every $x \in P$.
  Similarly, $\phi$ is an \textit{ascending closure operator} if $\phi^2 = \phi$ and $\phi(x) \geq x$ for every $x \in P$.
\end{definition}

It is well-known in topological combinatorics that ascending and descending closure operators induce strong deformation retracts.

\begin{definition}
  Given a poset $P$, $\Delta(P)$ is the \textit{order complex} of $P$, the simplicial complex whose simplices are the chains of $P$.
  Moreover, if $P$ is the face poset of a polyhedral cell complex $X$, then $\F(\Delta(P)) = Bd \ X$.
\end{definition}

\begin{theorem} (Kozlov~\cite{KozlovSimpleFold}) \label{Koperator}
  Let $P$ be a poset, and let $\phi$ be a descending closure operator.
  Then, $\Delta(P)$ collapses onto $\Delta(\phi(P))$.
  By symmetry, the same is true for an ascending closure operator.
\end{theorem}

\begin{proof}[Proof of Theorem \ref{posetfold}]

  First, we show that $Bd \ \Hom(P,Q)$ collapses onto $Bd \ \Hom(P-x,Q)$.
  Identify $\F(\Hom(P-x,Q))$ with the subposet of $\F(\Hom(P,Q))$ consisting of all $\eta$ such that $\eta(x) = \eta(y)$.
  Let $X$ be the subposet consisting of all $\eta \in \F(\Hom(P,Q))$ satisfying $\eta(x) \supseteq \eta(y)$.
  Then, $\F(\Hom(P-x,Q)) \subseteq X \subseteq \F(\Hom(P,Q))$.
  Consider order-preserving maps
  $$\F(\Hom(P,Q)) \stackrel{\alpha}{\rightarrow} X \stackrel{\beta}{\rightarrow} \F(\Hom(P-x,Q)),$$
  defined by
  $$\alpha \eta(z) = \left\{ \begin{array}{ll} \eta(y) \cup \eta(x), & \mathrm{for} \ z = x; \\ \eta(z), &\mathrm{otherwise;} \end{array} \right. \qquad
  \beta \eta(z) = \left\{ \begin{array}{ll} \eta(y), & \mathrm{for} \ z = x; \\ \eta(z), &\mathrm{otherwise;} \end{array} \right.$$
  for all $z \in P$.
  Maps $\alpha$ and $\beta$ are well defined because $P-x$ is a fold of $P$.
  It is straightforward to verify that $\alpha$ is an ascending closure operator and $\beta$ is a descending closure operator.
  Since the image of $\beta\circ\alpha$ is $\F(\Hom(P-x,Q)$, the result follows from Theorem~\ref{Koperator}.

  Next, we show that $\Hom(Q,P)$ collapses onto $\Hom(Q,P-x)$ by presenting a sequence of elementary collapses.
  Let $Q = \{q_1, q_2, \dots, q_t\}$.
  For $\eta \in \F(\Hom(Q,P))\setminus \F(\Hom(Q,P-x))$, let $1~\leq~i(\eta)~\leq~t$ be the minimal index such that $x \in \eta(q_{i(\eta)})$.
  Write $\F(\Hom(Q,P))$ as a disjoint union $A \uplus B \uplus \F(\Hom(Q,P-x))$, defined as follows: for $\eta \in A \cup B$, we have $\eta \in A$ if $y \notin \eta(q_{i(\eta)})$ while $\eta \in B$ otherwise.

  There is a bijection $\phi: A \rightarrow B$ which adds $y$ to $\eta(q_{i(\eta)})$ without changing the other values of $\eta$.
  Adding $y$ to $\eta(q_{i(\eta)})$ yields an element in $\F(\Hom(Q,P))$ since $P-x$ is a fold of $P$.
  Clearly, $\phi(\alpha)$ covers $\alpha$ for all $\alpha \in A$.
  We take the set $\{(\alpha, \phi(\alpha)) \ | \ \alpha \in A\}$ to be our collection of the elementary collapses.
  These are ordered lexicographically after the pairs of integers $(i(\alpha),- \mathrm{dim} \ \alpha)$.

  Let us see that these collapses can be performed in this lexicographic order.
  Take $\eta > \alpha$, $\eta \neq \phi(\alpha)$.
  Assume $i(\eta) = i(\alpha)$.
  If $\eta \in B$, then $\eta = \phi(\tilde{\alpha})$, $i(\tilde{\alpha}) = i(\alpha)$, and dim $\tilde{\alpha} > \mathrm{dim} \ \alpha$.
  Otherwise $\eta \in A$ and dim $\eta > \mathrm{dim} \ \alpha$.
  The third possibility is that $i(\eta) < i(\alpha)$.
  In either case, $\eta$ has been removed before $\alpha$.
\end{proof}

Consequently, certain poset elements with an appropriate neighborhood property can be ignored when examining $\Hom(P,Q)$.


\section{Complexes of maximal chains in a graded poset}\label{sec:maxchains}

We now turn our attention to $\Hom(Q,P)$ where $Q$ is a chain. 
There are two motivations for this.
First, from the perspective of folds discussed in the previous section, chains are primitive objects since they admit no folds.
Second, for a graded poset $P$, the set of maximal chains is known to play multiple important roles.
For example, the facets of the order complex of $P$ are precisely the maximal chains in $P$, and techniques such as EL-labelings frequently require a detailed study of maximal chains in $P$.
Thus, the rest of this paper will focus on this case, starting with the following definition.

\begin{definition}
  Let $C_m$ denote the chain $0<1<2<\cdots <m$ of rank $m$.
  Let $P$ be a graded poset of rank $m$, and define
  \[
    \Hom(P):=\Hom(C_m,P) \, .
  \]
\end{definition}

Thus, a vertex of $\Hom(P)$ corresponds to a maximal chain in $P$ realized as a strict map from $C_m$ to $P$.
Given a multihomomorphism $\eta$ corresponding to a cell in $\Hom(P)$, we will represent $\eta$ as the vector with set-valued entries $(\eta(0),\eta(1),\eta(2),\ldots,\eta(m))$, as we did in Example~\ref{ex:chainbooleanhom}.
In the case that $P$ is a graded lattice, we must have that $\eta(0)=\{\widehat{0}_P\}$ and $\eta(m)=\{\widehat{1}_P\}$.


\begin{example}\label{ex:b6}
  In $\Hom(B_6)$, we can construct the multihomomorphism 
  \[
    \eta = (\emptyset,\{1,2\},12,\{123,124\},1234,\{12345,12346\},123456) \, ,
  \]
  which corresponds to a three-dimensional cell that is the product of three edges, hence a cube.
\end{example}



\subsection{Distributive lattices}

For a finite poset $P$, let $J(P)$ denote the lattice of lower order ideals in $P$.
Recall Birkoff's foundational result~\cite[Theorem 17.3]{BirkhoffDistributive} that every finite distributive lattice $L$ is isomorphic to $J(P)$ for some poset $P$.
We will assume in this paper that every finite distributive lattice $L$ is presented as $L=J(P)$ for some $P$.

For such an $L$, a maximal chain $c$ is of the form
\[
  c = \{\emptyset = I_0\subset I_1\subset I_2\subset \cdots \subset I_{|P|}=P\}\, ,
\]
where for each $j$ we have  $I_j\subseteq P$, $|I_j|=j$, and there exists some $a_j\in P$ such that $I_j\setminus I_{j-1}=\{a_j\}$.
As a result, we can represent $c$ as the permutation in $\mathfrak{S}_P$ given by
\[
  c = a_1a_2a_3\cdots a_{|P|}.
\]

\begin{theorem}\label{thm:distcubical}
  Let $L=J(P)$ be a distributive lattice.
  If $\eta\in \Hom(L)$, then $|\eta(j)|$ is equal to either $1$ or $2$ for all $j$.
  Furthermore, if $|\eta(j)|=2$, then $|\eta(j-1)|=1$ when $j\geq 1$ and $|\eta(j+1)|=1$ when $j\leq m-1$.
  Hence, $\Hom(L)$ is a cubical complex.
\end{theorem}

\begin{proof}
  Suppose $\eta \in \Hom(L)$ with $\eta(j)=\{I_j^1,I_j^2,\ldots,I_j^k\}$.
  Setting $I_j^1\cap I_j^2=I$, we must have $|I|=j-1$, and thus $\eta(j-1)=\{I\}$ with $I\subset I_j^i$ for all $i$ in order for the vertices encoded by $\eta$ to be maximal chains in $L$.
  This implies that if $|\eta(j)|=k>2$, any $J\in \eta(j+1)$ satisfies $j+2\leq |\cup_{i=1}^kI_j^i|=|\bigvee_{i=1}^kI_j^i|\leq |J|=j+1$, a contradiction.
  Next, assume that $\eta(j)=\{I_j,I_j'\}$.
  As before, if $I_j\cap I_j'=I$ then we have $\eta(j-1)=\{I\}$.
  Thus, there exist $a_j,a_j'\in P$ such that $I_j\setminus I=\{a_j\}$ and $I_j'\setminus I=\{a_j'\}$, and again by maximality of the chains encoded by $\eta$ we must have that $\eta(j+1)=\{I_j\cup I_j'\}=\{I\cup \{a_j,a_j'\}\}$.
  Finally, since $\eta(j)$ has size $1$ or $2$ for all $j$, the cell in $\Hom(L)$ corresponding to $\eta$ is a product of points and edges, and thus $\Hom(L)$ is a cubical complex.
\end{proof}

The proof of Theorem~\ref{thm:distcubical} inspires the following useful notation to encode cells $\eta\in \Hom(L)$.
Let $\veps$ be a linear extension of $P$, i,e. a total order that is compatible with the partial order for $P$.
We encode $\eta$ as a parenthesized permutation $c_\eta$ where the $j$-th element in $c_\eta$ is $a_j$ if $\eta(j)=\{I_j\}$ and $\eta(j-1)=\{I_{j-1}\}$ with $I_j\setminus I_{j-1}=\{a_j\}$, and the $j$-th and $j+1$-st element of $c_\eta$ are $a_j$ and $a_j'$, enclosed in parentheses, if $\eta(j)=\{I_j,I_j'\}$ and $\eta(j-1)=\{I_{j-1}\}$ with $I_j\setminus I_{j-1}=\{a_j\}$ and $I_j'\setminus I_{j-1}=\{a_j'\}$.
We adopt the convention that any pair of elements contained in parentheses are listed in descending order with respect to $\veps$.
For example, given the cell from Example~\ref{ex:b6} with $\eta = (\emptyset,\{1,2\},12,\{123,124\},1234,\{12345,12346\},123456)$, the corresponding parenthesized permutation is
\[
  c_\eta=(21)(43)(65) \, .
\]

\begin{definition}
  Elements of $c_\eta$ that are not enclosed in parentheses in the permutation representation of $\eta$ are called \emph{free}, while any parenthesized elements are \emph{joined}.
\end{definition}

\begin{figure}
  \centering
  \begin{tikzpicture} [scale=0.25]
    \coordinate (1234) at (15,0);
    \coordinate (1243) at (20,1.25);
    \coordinate (1324) at (16,5);
    \coordinate (1342) at (21,10);
    \coordinate (1423) at (25,6.5);
    \coordinate (1432) at (26,11);
    \coordinate (2134) at (7,1.25);
    \coordinate (2143) at (13,2.25);
    \coordinate (2314) at (2,6.5);
    \coordinate (2341) at (1,12);
    \coordinate (2413) at (12,8);
    \coordinate (2431) at (7,13);
    \coordinate (3124) at (11,10);
    \coordinate (3142) at (16,15);
    \coordinate (3214) at (3,11.25);
    \coordinate (3241) at (2,17);
    \coordinate (3412) at (15,21.75);
    \coordinate (3421) at (7,23);
    \coordinate (4123) at (24,11.5);
    \coordinate (4132) at (25,17);
    \coordinate (4213) at (17,13); 
    \coordinate (4312) at (20,23);
    \coordinate (4231) at (12,18);
    \coordinate (4321) at (13,24);
    \draw[fill=lightgray] (1234)--(1243)--(2143)--(2134)--cycle;
    \draw[fill=lightgray] (4231)--(4213)--(2413)--(2431)--cycle;
    \draw[fill=lightgray] (1423)--(1432)--(4132)--(4123)--cycle;
    \draw [] (1234)--(1243)--(1423)--(1432)--(1342)--(1324)--cycle;
    \draw [] (1234)--(1324)--(3124)--(3214)--(2314)--(2134)--cycle;
    \draw [] (2134)--(2143)--(2413)--(2431)--(2341)--(2314)--cycle;
    \draw [] (1243)--(2143)--(2413)--(4213)--(4123)--(1423)--cycle;
    \draw [] (3142)--(3124)--(3214)--(3241)--(3421)--(3412)--cycle;
    \draw [] (3142)--(1342)--(1432)--(4132)--(4312)--(3412)--cycle;
    \draw [] (4231)--(2431)--(2341)--(3241)--(3421)--(4321)--cycle;
    \draw [] (4321)--(4312)--(4132)--(4123)--(4213)--(4231)--cycle;
    \draw[fill=lightgray] (3421)--(4321)--(4312)--(3412)--cycle;
    \draw[fill=lightgray] (3241)--(3214)--(2314)--(2341)--cycle;
    \draw[fill=lightgray] (1324)--(1342)--(3142)--(3124)--cycle;
    \node [below] at (1234) {\small 1234};
    \node [below] at (1243) {\small 1243};
    \node [right] at (1324) {\small 1324};
    \node [below, xshift=7pt] at (1342) {\small 1342};
    \node [right] at (1423) {\small 1423};
    \node [right] at (1432) {\small 1432};
    \node [below] at (2134) {\small 2134};
    \node [above left] at (2143) {\small 2143};
    \node [left] at (2314) {\small 2314};
    \node [left] at (2341) {\small 2341};
    \node [left] at (2413) {\small 2413};
    \node [above left] at (2431) {\small 2431};
    \node [right] at (3124) {\small 3124};
    \node [right] at (3142) {\small 3142};
    \node [xshift=11pt, yshift=-10pt] at (3214) {\small 3214};
    \node [left] at (3241) {\small 3241};
    \node [below right] at (3412) {\small 3412};
    \node [above] at (3421) {\small 3421};
    \node [xshift=-12pt, yshift=10pt] at (4123) {\small 4123};
    \node [right] at (4132) {\small 4132};
    \node [left] at (4231) {\small 4231};
    \node [right] at (4312) {\small 4312};
    \node [above] at (4321) {\small 4321};
  \end{tikzpicture}
  \caption{$\Hom(B_4)$}
  \label{fig:perm4}
\end{figure}

\begin{example}\label{ex:booleancells}
  Let $B_n$ denote the Boolean algebra on $n$ elements, i.e the power set of $[n]$ ordered by set inclusion.
  Note that $B_n$ is isomorphic to the product of $n$ copies of $C_1$, a chain with two elements.
	Also note that $B_n=J([n])$ where $[n]$ indicates an antichain with $n$ elements.
  Thus, the set of vertices of $\Hom(B_n)$ corresponds exactly to the elements of $\mathfrak{S}_n$.
  Two vertices in $\Hom(B_n)$ are endpoints of a common edge in $\Hom(B_n)$ if they differ by a single transposition of adjacent entries, that is $\{ \sigma, \tau \}\subseteq \mathfrak{S}_n$ form a 1-cell in $\Hom(B_n)$ if $\sigma_i = \tau_{i+1}$ and $\sigma_{i+1} = \tau_i$ for some $1 \leq i \leq n-1$ while $\sigma_j = \tau_j$ for $j \notin \{i, i+1\}$.
  Observe that this implies that $\Hom(B_n)$ lies in the boundary of the permutohedron of order $n$, see Figure~\ref{fig:perm4}.
  Without loss of generality, assume that $\sigma$ is lexicographically earlier than $\tau$ under the usual ordering of the positive integers.
  Then, we can denote the above transposition by $(\sigma_{i+1} \sigma_i)$, written in decreasing order following our conventions.
  We denote the corresponding  1-cell as
  \[
    \sigma_1 \cdots \sigma_{i-1} (\sigma_{i+1} \sigma_i) \sigma_{i+2} \cdots \sigma_n \, .
  \]
  For example, the homomorphisms given by permutations $35142$ and $31542$ are adjacent in $\Hom(B_5)$ via the transposition $(51)$, and so the corresponding $1$-cell (where on the left we write a parenthesized permutation, and on the right we denote a vector of subsets) is
  \[
    3(51)42 = (\emptyset,3,\{13,35\},135,1345,12345) \, .
  \]
  The elements $2$, $3$, and $4$ are free while and $1$ and $5$ are joined.
  As another example, the expression $(64)5(32)(71)$ denotes the 3-cell in $\Hom(B_7)$ with vertex set the homomorphisms given by the permutations
  $$\{ 4652317, 4652371, 4653217, 4653271, 6452317, 6452371, 6453217, 6453271 \}.$$
  The only free element in this example is $5$, while each pair of $6$ and $4$, $3$ and $2$, and $7$ and $1$ are joined.
\end{example}

\begin{remark}
  It follows from Example~\ref{ex:booleancells} that $\Hom(B_n)$ is isomorphic to the subcomplex of the boundary of the permutahedron formed by all cubical faces.
  A generalization of this complex has been previous studied by Severs and White~\cite{severswhite}, and their paper contains proofs of the optimal discrete Morse matchings in this paper for the case of $\Hom(B_n)$, though in a very different language.
\end{remark}


\subsection{Incidence numbers in $\Hom(L)$}

Recall that a general homomorphism complex arises as a subcomplex of $P(A,B)=\prod_{i\in A}\Delta_B$ for some finite sets $A$ and $B$, which we can (after linearly ordering and renaming the elements in $A$) write as $P(A,B)=\prod_{i=1}^m\Delta_B$.
If we assume that the elements of $B$ are endowed with a total order, then we have $P(A,B)$ is a product of oriented simplices.
Thus, the cellular chain complex for $P(A,B)$ is obtained as a tensor product of copies of the chain complex for $\Delta_B$~\cite[Section 3B]{Hatcher}, which allows us to determine the incidence numbers as follows.
Let $\eta \in P(A,B)$ and $t\in \{1,\ldots,m\}$ with $|\eta(t)|\geq 2$.
For $\eta(t)=\{v_0^t,\dots,v^t_{|\eta(t)|}\}$ and $0\leq \ell \leq |\eta(t)|$, consider the face of $\eta$ given by
\[
  \tau = (\eta(1),\ldots,\eta(t-1),\eta(t)\setminus\{v_\ell^t\},\eta(t+1),\ldots,\eta(m)) \, ,
\]
for which the incidence number is obtained via the expression
\begin{equation}\label{eqn:incidence}
  [\tau:\eta]=(-1)^{t-1+\ell+\sum_{j=1}^{t-1}|\eta(j)|} \, .
\end{equation}
The cellular boundary map $\partial$ for $P(A,B)$ is defined as
\[
  \partial(\eta)=\sum_{\tau\text{ a facet of }\eta}[\tau:\eta]\tau \, .
\]
In the case of a distributive lattice $L=J(P)$, we want to be able to determine the incidence numbers using our representation $c_\eta$ of a cell $\eta\in \Hom(L)$.
The complex $\Hom(L)$ arises as a subcomplex of $\prod_{t=1}^{|P|}\Delta_{2^P}$, where $2^P$ denotes the set of all subsets of $P$.
For a linear extension $\veps$ of $P$, we list the elements of $2^P$ in graded lexicographic order.
Let $c_\eta=w(\beta\alpha)u$ where $w$ and $u$ are parenthesized permutations such that $w$ contains $t-1$ joined pairs, and $\beta$ and $\alpha$ are a joined pair in $c_\eta$.
Defining $c_\eta^\alpha=w\alpha\beta u$ and $c_\eta^\beta=w\beta\alpha u$, it follows from~\eqref{eqn:incidence} that
\begin{equation}\label{eqn:pairincidence}
  [c_\eta^\alpha:c_\eta]=(-1)^{t-1} \text{ and }[c_\eta^\beta:c_\eta]=(-1)^{t} \, .
\end{equation}

\begin{example}
  In $\Hom(B_7)$, the parenthesized permutation $c_\eta=(64)5(32)(71)$ corresponds to the cell
  \[
    \eta=(\emptyset,\{4,6\},46,456,\{2456,3456\},23456,\{123456,234567\},1234567) \, .
  \]
  If $(\beta\alpha)=(32)$, then $c_\eta^\beta=(64)532(71)$ corresponds to the cell
  \[
    \tau=(\emptyset,\{4,6\},46,456,3456,23456,\{123456,234567\},1234567) \, .
  \]
  Thus, $[\tau:\eta]=[c_\eta^\beta:c_\eta]=-1$.
  A similar analysis shows that $[c_\eta^\alpha:c_\eta]=1$.
\end{example}


\section{Products of Chains}\label{sec:productsofchains}

For the remainder of this paper, we turn our attention to those distributive lattices $L$ that arise as products of chains.
For $1\leq i_1\leq i_2\leq \cdots \leq i_n$ integers, we define $\bi=(i_1,\ldots,i_n)$ and 
\[
  \Hom(\bi):=\Hom\left(\prod_{j=1}^nC_{i_j}\right) \, .
\]
Note that when $i_j=1$ for all $j$, we have that $\prod_{j=1}^n C_{1}$ is isomorphic to $B_n$.
Our main result in this section is the following.

\begin{theorem}\label{thm:mainchain}
  For $\bi=(i_1,\ldots,i_n)$, there exists an optimal discrete Morse matching on the face poset of $\Hom(\bi)$ where the incidence number in the resulting Morse complex for any pair of critical cells is zero.
  Thus, the homology groups of $\Hom(\bi)$ are free and the rank of $H_k(\Hom(\bi);\ZZ)$ is equal to the number of critical cells of dimension $k$ in the matching.
  Furthermore, the critical cells in the matching are in bijection with permutations $w$ of the multiset $\{1^{i_1},2^{i_2},\ldots,n^{i_n}\}$ whose descent set $\Des(w)$ can be written as a disjoint union
  \begin{equation}\label{eqn:desunion}
    \Des(w) = \biguplus_{t} \{m_t,m_t+1,\ldots,m_t+q_t\}
  \end{equation}
  where each $q_t$ is of the form $3j+1$ or $3j+2$ and $\Des(w) \cap \bigcup_t \{m_t-1, m_t+q_t+1\} = \emptyset$.
  The dimension of the critical cell corresponding to $w$ is 
  \[
    \sum_t\left\lceil \frac{q_t}{3} \right\rceil \, .
  \]
\end{theorem}

The remaining three subsections of this paper contain the proof of Theorem~\ref{thm:mainchain}.
We will prove the theorem in three parts: first we will define a partial matching and prove it is acyclic, second we will prove that the claimed critical cell bijection holds, and third we will prove that our acyclic partial matching is optimal.
Before proceeding to the proofs for each of these parts, we introduce notation that we will require.

Note that $\prod_{j=1}^n C_{i_j}=J\left(\biguplus_{j=1}^nC_{i_j-1}\right)$.
If we were to use our earlier convention, we would set $P=\biguplus_{j=1}^nC_{i_j-1}$ and consider a linear extension of the elements of $P$ in order to write our parenthesized words for cells.
However, in the special case of $\biguplus_{j=1}^nC_{i_j-1}$, we can use a simplified representation of our parenthesized words.
Observe that every strict poset map $\eta$ from $C_{\sum_{j=1}^ni_j}$ into $\prod_{j=1}^n C_{i_j}$ corresponds to a word $\tilde{c}_\eta$ that is a permutation of the multiset $\{ 1^{i_1}, 2^{i_2}, \dots, n^{i_n} \}$, as follows.
If $\eta(j)=I_j$, $\eta(j-1)=I_{j-1}$, and $I_j\setminus I_{j-1}=\{a_j\}\subset C_{i_t-1}$, then we set the $j$-th element of the permutation $\tilde{c}_{\eta}$ equal to $t$ instead of $a_j$.
Because each $C_{i_j}$ is a total order, we can recover $c_\eta$ from $\tilde{c}_\eta$ by counting the number of occurrences of $t$ to the left of the occurrence in the $j$-th position of  $\tilde{c}_\eta$, from which we can recover the element $a_j$.

\begin{example}
  The permutation $\tilde{c}_\eta=123213$ corresponds to some vertex $\eta$ in $\Hom(2,2,2)$.
  If $C_2\times C_2\times C_2$ is represented as $J\left(\{r<s\}\uplus \{a<b\}\uplus \{x<y\}\right)$, then the corresponding maximal chain in $C_2\times C_2\times C_2$ is
  \[
    \eta = (\emptyset,r,ra,rax,raxb,raxbs,raxbsy) \, .
  \]
\end{example}

Two vertices $\sigma$ and $\tau$ are adjacent in $\Hom(\mathbf{i})$ if we can swap two adjacent distinct entries in the the word corresponding to $\sigma$ to get the word corresponding to $\tau$ and vice versa.
Note that the word does not change if the adjacent entries are identical.
As we did previously, we denote the 1-dimensional cell connecting $\sigma$ and $\tau$ by putting the transpositional pair of entries in a set of parentheses, with the two entries written in decreasing order.
We extend this convention to denote a $k$-dimensional cell in $\Hom(\bi)$ by a word from the multiset $\{ 1^{i_1}, 2^{i_2}, \dots, n^{i_n} \}$ with $k$ non-overlapping parenthesized pairs where the entries in each pair are distinct and written in decreasing order.

\begin{definition}Given a $k$-dimensional cell $\sigma\in \Hom(\bi)$, the word obtained by removing the parentheses from $\sigma$ but keeping all of the entries in the same order will be called the \textit{underlying word of $\sigma$}.
  For a fixed vector $\mathbf{i}$ and corresponding multiset $\{ 1^{i_1}, 2^{i_2}, \dots, n^{i_n} \}$, consider an arbitrary cell $\sigma$ and its underlying word.
  Let $r_s$ denote the $s$-th $r$ when reading the underlying word from left to right, and let $j_{r,s}$ denote the position of $r_s$.
\end{definition}

\begin{example}
  The cell $123(21)3$ is contained in $\Hom(2,2,2)$, has underlying word $123213$, and has the following $j$-values:

  \begin{tabular}{lcl}
    $j_{1,1} = 1$, the location of the first 1 & \hspace{0.5in} & $j_{1,2} = 5$, the location of the second 1 \\
    $j_{2,1} = 2$, the location of the first 2 & \hspace{0.5in} & $j_{2,2} = 4$, the location of the second 2 \\
    $j_{3,1} = 3$, the location of the first 3 & \hspace{0.5in} & $j_{3,2} = 6$, the location of the second 3 \\
  \end{tabular}
\end{example}

\subsection{The discrete Morse matching}

Let $\ell := \sum_{k=1}^{n} i_k$.
Let $C_{\ell}^{op}$ denote the chain poset on $[\ell] := \{1,2,\dots,\ell\}$ with the order $\ell<\ell-1<\cdots <2<1$, and let $M$ denote the poset of two elements $a$ and $b$ with $a < b$.
Let $\mathcal{F}(\Hom(\bi))$ denote the face poset of the complex $\Hom(\bi)$.
We define an acyclic partial matching on $\mathcal{F}(\Hom(\bi))$ according to the following algorithm.

\begin{algorithm}\label{MatchAlg}

  INITIALIZE the following: \qquad $r := n$ \qquad $s := i_n$.

  Define $\Omega_{r,s} := \mathcal{F}(\Hom(\mathbf{i}))$  and $U_{r,s} := \emptyset$.

  \bigskip

  STEP 1: Define a poset map $\phi_{r,s}$ from $\Omega_{r,s}$ to $C_{\ell}^{op}$ where $\sigma$ is sent to the position of $r_s$, that is $\phi_{r,s}(\sigma) = j_{r,s}$.
  This map is order-preserving by Claim \ref{PhiProof}.

  \bigskip

  STEP 2: Define a poset map $\rho_{r,s}$ from $\phi_{r,s}^{-1}(j_{r,s}) \subseteq \Omega_{r,s}$ to $M$ as follows:
  \begin{itemize}
  \item $\rho_{r,s}(\sigma) = a$ if the following conditions hold:
    \begin{enumerate}
    \item If $\sigma_{j_{r,s}-1}$ is free, then $r_s \geq \sigma_{j_{r,s}-1}$.  
    \item $r_s > \sigma_{j_{r,s}+1}$.  
    \item Either $r_s$ and $\sigma_{j_{r,s}+1}$ are both free or joined together.  
    \end{enumerate}
  \item $\rho_{r,s}(\sigma) = b$ otherwise.
  \end{itemize}
  This map is order-preserving by Claim \ref{RhoProof}.

  \bigskip

  STEP 3: Perfectly match the elements of $\rho_{r,s}^{-1}(a) \subseteq \phi_{r,s}^{-1}(j_{r,s}) \subseteq \Omega_{r,s}$ by pairing the cells where $r_s$ and $\sigma_{j_{r,s}+1}$ are both free with the corresponding cell where $r_s$ and $\sigma_{j_{r,s}+1}$ are joined together.
  This matching is acyclic since $r_s$ can be in at most one transposition for any given cell.

  \bigskip

  STEP 4: 
  If $s > 1$, redefine $s := s-1$.
  Now, define $\Omega_{r,s} := \rho_{r,s+1}^{-1}(b) \subseteq \phi_{r,s+1}^{-1}(j_{r,s+1}) \subseteq \Omega_{r,s+1}$ and $U_{r,s} := \{r_{s+1}\} \cup U_{r,s+1}$, then return to Step 1.

  If $s = 1$ and $r > 1$, redefine $r := r-1$ and $s := i_r$.
  Now, define $\Omega_{r,s} := \rho_{r+1,1}^{-1}(b) \subseteq \phi_{r+1,1}^{-1}(j_{r+1,1}) \subseteq \Omega_{r+1,1}$ and $U_{r,s} := \{(r+1)_{1}\} \cup U_{r+1,1}$, then return to Step 1.

  If $s = 1$ and $r = 1$, STOP.

\end{algorithm}

Before proving Claims~\ref{PhiProof} and~\ref{RhoProof} and Theorem~\ref{thm:matchalg}, we will go through an example illustrating the action of our algorithm on a specific cell $\sigma$.

\begin{example}
  Let $\sigma=(21)1(32)344\in\Hom(2,2,2,2)$.
  We begin by setting $r=4$ and $s=2$, with $\Omega_{4,2}=\mathcal{F}(\Hom(2,2,2,2))$ and $U_{4,2}=\emptyset$.
  We have that $\phi_{4,2}(\sigma)=8$, thus $\sigma\in \phi_{4,2}^{-1}(8)$.
  Since $\sigma_{8+1}=\sigma_9$ is not defined, we have $\rho_{4,2}(\sigma)=b$.

  Next we set $r=4$ and $s=1$, with $\Omega_{4,1}=\rho_{4,1}^{-1}(b)$, hence $\sigma\in \Omega_{4,1}$ and $U_{4,1}=\{4_2\}$.
  We have that $\phi_{4,1}(\sigma)=7$.
  Since $4_1=4$ is not strictly greater than $\sigma_8=4$, we have that $\rho_{4,1}(\sigma)=b$.

  Next, we set $r=3$ and $s=2$, with $\Omega_{3,2}=\rho_{4,1}^{-1}(b)$ and $U_{3,2}=\{4_1,4_2\}$.
  Then $\phi_{3,2}(\sigma)=6$ and since $3$ is not greater than $4$, we have $\rho_{3,2}(\sigma)=b$.

  Again, we set $r=3$ and $s=1$, with $\Omega_{3,1}=\rho_{3,2}^{-1}(b)$ and $U_{3,1}=\{3_2,4_1,4_2\}$.
  We have that $\phi_{3,1}(\sigma)=4$, and also we have $\rho_{3,1}(\sigma)=a$ since $\sigma_3=1\leq 3$, $3>\sigma_5=2$, and $(\sigma_4\sigma_5)=(32)$ are joined.
  In $\rho_{3,1}^{-1}(a)$, we match $(21)1(32)344$ with $(21)132344$.
\end{example}

\begin{claim}\label{PhiProof}
  The poset maps $\phi_{r,s}$ defined in Step 1 of Algorithm \ref{MatchAlg} are order-preserving within their corresponding $\Omega_{r,s}$ domains.
\end{claim}

\begin{proof}
  Suppose $\sigma < \tau$ in $\Omega_{r,s}$ with $\phi_{r,s}(\tau) = j \in [\ell]$.
  If we release a transposition that does not include $r_s$, the position of $r_s$ will be preserved.
  If $r_s$ is in a transposition of the form $(r_s \tau_{j+1})$, then releasing that transposition in either order will preserve the position of $r_s$ or increase its position by one.
  Suppose that $r_s$ is in a transposition of the form $(\tau_{j-1} r_s)$.
  By convention, $\tau_{j-1} > r_s$, which implies $\tau_{j-1} \in U_{r,s}$, a consequence of Step 4 in Algorithm \ref{MatchAlg}.
  By the iterative revision of $\Omega_{r,s}$ in Step 4 of Algorithm \ref{MatchAlg}, the elements of $U_{r,s}$ have prescribed positions in the underlying words of the cells in $\Omega_{r,s}$.
  Hence, the position of $\tau_{j-1}$ must be fixed in all elements of $\Omega_{r,s}$, which implies that the transposition $(\tau_{j-1} r_s)$ can only be released in the order $\tau_{j-1} r_s$ to stay within $\Omega_{r,s}$.
  So, the position of $r_s$ is preserved in this third case.
  In all three cases, $\phi_{r,s}(\sigma) \geq j$.
  Since cells with larger position values for $r_s$ map lower on $C_{\ell}^{op}$, we see that $\phi_{r,s}$ is weakly order preserving.
\end{proof}

\begin{claim}\label{RhoProof}
  The poset maps $\rho_{r,s}$ defined in Step 2 of Algorithm \ref{MatchAlg} are order-preserving within their corresponding $\phi_{r,s}^{-1}(j_{r,s}) \subseteq \Omega_{r,s}$ domains.
\end{claim}

\begin{proof}
  Suppose $\sigma$ is covered by $\tau$ in $\phi_{r,s}^{-1}(j_{r,s}) \in \Omega_{r,s}$ with $\rho_{r,s}(\tau) = a$.
  If both $\tau_{j_{r,s}}$ and $\tau_{j_{r,s}+1}$ are free in $\tau$, then they are both free in $\sigma$ since transpositions are inherited upward and free entries are inherited downward.
  So, $\rho_{r,s}(\sigma) = a$ here.
  Suppose instead that $\tau_{j_{r,s}}$ and $\tau_{j_{r,s}+1}$ are joined.
  If they are still joined in $\sigma$, then $\rho_{r,s}(\sigma) = a$.
  If they are not joined in $\sigma$, then they are both free in $\sigma$.
  However, there are two children of $\tau$ where the entries at positions $j_{r,s}$ and $j_{r,s}+1$ are free.
  We know that $\sigma$ must be the child where the entries are in decreasing order because the child of $\tau$ where the entries are in increasing order is in the fiber $\phi_{r,s}^{-1}(j_{r,s}+1)$, not $\phi_{r,s}^{-1}(j_{r,s})$.
  This means that $\rho_{r,s}(\sigma) = a$ here as well.
  Hence, we must have $\rho_{r,s}(\sigma) = a$, which implies that $\rho_{r,s}$ is weakly order preserving.
\end{proof}

\begin{theorem}\label{thm:matchalg}
  The matching defined by Algorithm \ref{MatchAlg} is an acyclic partial matching.
\end{theorem}

\begin{proof}
  The claim follows from the Patchwork Theorem since the sub-matchings defined in Step 3 are each acyclic.
\end{proof}


\subsection{The critical cell bijection}

Let $\sigma$ denote an arbitrary critical $k$-cell in the reduced complex of $\Hom(\mathbf{i})$, let $w$ denote its underlying word which is a permutation of the multiset $\{1^{i_1},2^{i_2},\ldots,n^{i_n}\}$.
Define
\[
  M := \{j \in \Des(w) : j-1 \notin \Des(w)\} \,
\]
and let $m_1, m_2, \dots, m_s$ denote the elements of $M$ in increasing order where $s = |M|$.
We first prove the following three claims to be used in the main bijection proof.

\begin{claim}\label{StartDesClaim}
  If $j \in M$, then $\sigma_j$ is free.
\end{claim}

\begin{proof}
  Since $j-1 \notin \Des(w)$ and joined pairs correspond to a descent, we conclude that $\sigma_{j-1}$ and $\sigma_j$ are not joined.
  We will show that $\sigma_j$ is not joined with $\sigma_{j+1}$.
  For the sake of contradiction, assume that $\sigma_j$ and $\sigma_{j+1}$ are joined.
  Since $\sigma_j = r_s$ for some $r \in [n]$ and $1 \leq s \leq i_r$, we must have $j = j_{r,s}$ and $\sigma \in \phi_{r,s}^{-1}(j_{r,s}) \in \Omega_{r,s}$.
  Since $j-1 \notin \Des(w)$, the conclusion of Condition (1) in Step 2 of Algorithm~\ref{MatchAlg} holds.
  Since $j-1 \notin \Des(w)$ and joined pairs correspond to a descent, Conditions (2) and (3) in Step 2 of Algorithm~\ref{MatchAlg} hold as well.
  Hence, $\rho_{r,s}(\sigma) = a$, and $\sigma$ is matched by Step 3 of Algorithm~\ref{MatchAlg} on the $r,s$ loop.
  This contradicts that $\sigma$ is a critical cell.
\end{proof}

\begin{claim}\label{j+1Claim}
  If $\sigma_j$ is free and $j \in \Des(w)$, then $\sigma_{j+1}$ is joined with $\sigma_{j+2}$, and hence $j+1 \in \Des(w)$.
\end{claim}

\begin{proof}
  Suppose that $\sigma_j$ is free and $j \in \Des(w)$.
  We will argue that $\sigma_{j+1}$ cannot be free and, hence, must be joined with $\sigma_{j+2}$.
  Since $\sigma_j = r_s$ for some $r \in [n]$ and $1 \leq s \leq i_r$, we must have $j = j_{r,s}$ and $\sigma \in \phi_{r,s}^{-1}(j_{r,s}) \in \Omega_{r,s}$.
  Since $\sigma_j$ is free and $j \in \Des(w)$, we must have that either $j-1 \notin \Des(w)$ or $\sigma_{j-1}$ is not free.
  Either way, Condition (1) in Step 2 of Algorithm \ref{MatchAlg} holds.
  Since $j \in \Des(w)$, Condition (2) holds as well.
  However, since $\sigma$ is critical, we cannot have $\rho_{r,s}(\sigma) = a$.
  This means that Condition (3) cannot be satisfied, which means that $\sigma_{j+1}$ cannot be free.
  Hence, $\sigma_{j+1}$ must be joined with something, i.e. $\sigma_{j+2}$.
  It follows that $j+1 \in \Des(w)$ since we write joined pairs in descending order.
\end{proof}

\begin{claim}\label{j+3Claim}
  Suppose $\sigma_j$ is free and $j \in \Des(w)$.
  If $j+3 \in \Des(w)$, then $\sigma_{j+3}$ is free.
\end{claim}

\begin{proof}
  Suppose that $\sigma_j$ is free and $j \in \Des(w)$.
  Claim \ref{j+1Claim} asserts that $\sigma_{j+1}$ is joined with $\sigma_{j+2}$ and that $j+1 \in \Des(w)$.
  Suppose that $j+3 \in \Des(w)$.
  We argue that $\sigma_{j+3}$ cannot be joined with $\sigma_{j+4}$.
  For the sake of contradiction, suppose that $\sigma_{j+3}$ and $\sigma_{j+4}$ are joined.
  Since $\sigma_{j+3} = r_s$ for some $r \in [n]$ and $1 \leq s \leq i_r$, we must have $j+3 = j_{r,s}$ and $\sigma \in \phi_{r,s}^{-1}(j_{r,s}) \in \Omega_{r,s}$.
  Since $\sigma_{j+3}$ is not free, Condition (1) in Step 2 of Algorithm \ref{MatchAlg} vacuously holds.
  Conditions (2) and (3) hold as a result of our hypothesis and the fact that joined elements are a descending pair.
  Hence, $\rho_{r,s}(\sigma) = a$ and $\sigma$ is matched by Step 3 of Algorithm~\ref{MatchAlg} on the $r,s$ loop.
  This contradicts that $\sigma$ is a critical cell, hence $\sigma_{j+3}$ must be free.
\end{proof}

We will now use the previous three claims to prove the bijection between the critical cells and the words with prescribed descent sets.

\begin{proof}[Proof of Critical Cell Bijection]
  Observe first that for every permutation $w$ of $\{1^{i_1},2^{i_2},\ldots,n^{i_n}\}$, the descent set $\Des(w)$ can be written as a disjoint union of intervals, each beginning with an element $m_t\in M$, i.e. there exist non-negative integers $q_t$ such that
  \begin{equation}
    \Des(w) = \biguplus_{t} \{m_t,m_t+1,\ldots,m_t+q_t\}
  \end{equation}
  where $\bigcup_t \{m_t-1, m_t+m_t+1\} \cap \Des(w) = \emptyset$.
  Thus, what we need to establish is that each $q_t$ is of the form $3k+1$ or $3k+2$, and we will be done.
  
  The result is proved via the following algorithm.
  We assume each $m_t$ is an element of $M$, i.e. $m_t \in \Des(w)$ but $m_t-1 \notin \Des(w)$.

  \begin{algorithm}

    INITIALIZE: $t:=1$ and $k:=0$.

    Step 1: WHILE $t \leq |M|$, observe that $m_t \in M \subseteq \Des(w)$ and $\sigma_{m_t}$ is free by Claim \ref{StartDesClaim}.

    Step 2: Observe that Claim~\ref{j+1Claim} asserts that that $m_t + 3k + 1 \in \Des(w)$.

    Step 3: If $m_t + 3k + 2 \notin \Des(w)$, then $q_t = 3k + 1$, which agrees with our description of the descent set of a critical cell.
    So, if $m_t + 3k + 2 \notin \Des(w)$, then increment $t$ by one and return to Step 1.
    Otherwise, continue to Step 4.

    Step 4: Since $m_t + 3k + 2 \in \Des(w)$, we consider $m_t + 3k + 3$.
    If $m_t + 3k + 3 \notin \Des(w)$, then $q_t = 3k + 2$, which agrees with our description of the descent set of a critical cell.
    So, if $m_t + 3k + 3 \notin \Des(w)$, then increment $t$ by one and return to Step 1.
    Otherwise, continue to Step 5.

    Step 5: Since $m_t + 3k + 3 \in \Des(w)$, $\sigma_{m_t + 3k + 3}$ is free by Claim \ref{j+3Claim}.
    Increment $k$ by one and return to Step 2.
		
  \end{algorithm}

  This algorithm shows that, for each $1 \leq t \leq |M|$ and $m_t\in M$, there exists some $q_t$ of the form $3k+1$ or $3k+2$ such that the set $\{m_t, m_t+1, \dots, m_t + q_t\} \subseteq \Des(w)$ while $\{m_t-1, m_t+q_t+1\} \cap \Des(w) = \emptyset$.

\end{proof}

\begin{example}
  Consider the critical cell $237(64)9(85)1$ with underlying word $w=237649851$.
  Then $M=\{3,6\}$ and $\Des(w)=\{3,4\}\uplus \{6,7,8\}$.
\end{example}

\begin{example}
  Consider the critical cell $2223(21)16(54)3(21)334(32)35$ which has underlying word $w=22232116543213343235$
  Then $M=\{4,8,16\}$ and
  \[
    \Des(w)=\{4,5,6\}\uplus \{8,9,10,11,12\}\uplus \{16,17\} \, .
  \]
\end{example}

\subsection{Optimality of the matching}

To prove that our matching is optimal, we will show that all the incidence numbers between critical cells in the resulting Morse complex are equal to zero.
Our method of proof is to produce a sign-reversing involution on the set of all alternating paths between any pair of critical cells in adjacent dimensions, where the sign that is reversed is the weight of that alternating path as defined by~\eqref{eqn:weight} in Theorem~\ref{thm:dmt}.
The following useful characterization of the critical cells is a direct consequence of Claim \ref{j+1Claim}.

\begin{claim}\label{claim:increasing}
  Consecutive free entries in a critical cell are weakly increasing; hence, any critical $k$-cell $\sigma$ must be of the form
	\[
		\sigma=w_1(\beta_1\alpha_1)w_2(\beta_2\alpha_2)w_3\cdots (\beta_{k}\alpha_{k})w_{k+1}
  \]
  where each $w_i$ is a weakly increasing multiset permutation of length $\ell_i$ where $w_{i,\ell_i}>\beta_i>\alpha_i$.
\end{claim}

\begin{proof}[Proof of Optimality]
  Let $\sigma$ be a critical $k$-cell and $\tau$ be a critical $(k-1)$-cell such that there exists an alternating path from $\sigma$ to $\tau$ given as follows:
  \begin{equation}\label{eqn:altpath}
    c = (\sigma,a_1,u(a_1),a_2,u(a_2),a_3,\ldots,a_t,u(a_t),\tau) \, .
  \end{equation}
  We will say that the \emph{length} of $c$ is $t$.
  By Claim~\ref{claim:increasing}, we know that $\tau$ is of the form
  \[
    \tau=w_1(\beta_1\alpha_1)w_2(\beta_2\alpha_2)w_3\cdots (\beta_{k-1}\alpha_{k-1})w_{k}
  \]
  where each $w_i$ is a weakly increasing multiset permutation of length $\ell_i$ where $w_{i,\ell_i}>\beta_i>\alpha_i$.

  In order for $\tau \subset u(a_t)$, we have that $u(a_t)$ contains all the joined pairs in $\tau$ plus one additional joined pair.
  Call this additional joined pair $(\beta,\alpha)$, where $\beta>\alpha$.
  Because $\tau$ is critical and $a_t$ is paired with $u(a_t)$ by joining $\alpha$ and $\beta$, it must be that $a_t$ is obtained from $\tau$ by swapping the positions of $\alpha$ and $\beta$ since these are the only elements in the parenthesized permutation that are not in the configuration of a critical cell.
  Extending this idea, we call a sequence $(a_i,u(a_i),a_{i+1})$ a \emph{swap} if $a_{i}$ is obtained from $a_{i+1}$ by exchanging two adjacent free entries in $a_{i+1}$.

  Now, for our alternating path $c$ defined in~\eqref{eqn:altpath}, let $j$ be the maximum index in $c$ such that $(a_j,u(a_j),a_{j+1})$ is not a swap.
  We first argue that $t\geq 2$, then that $j\leq t-1$.
  First, for any alternating path $c$, we have $t\geq 2$, since $a_1$ is obtained from $\sigma$ by releasing a joined pair, which introduces a pair of adjacent descents in a run of free entries of $a_1$ that must be reversed to return to critical position, i.e. satisfying Claim~\ref{claim:increasing}.
  Second, we must have that $u(a_t)$ is obtained from $\tau$ by joining two free elements of $\tau$ which are in increasing order, elements which are then reversed in order when placed in parentheses.
  All of the other joined pairs in $u(a_t)$ are in the configuration given in \ref{claim:increasing}, and thus $a_t$ must arise from a swap of the two newly joined elements in $\tau$.
  Hence, our choice of $j$ is well-defined.

  Since each $(a_i,u(a_i),a_{i+1})$ for $i\geq j+1$ is a swap, it follows that there exist permutations $\pi_i(w_i)$ of the weakly ascending multiset permutations in $\tau$ such that $a_j$ is of the form
  \[
    a_j=\pi_1(w_1)(\beta_1\alpha_1)\pi_2(w_2)(\beta_2\alpha_2)\pi_3(w_3)\cdots (\beta_{k-1}\alpha_{k-1})\pi_{k}(w_{k})\, .
  \]
  Furthermore, $u(a_j)$ is obtained by joining two of the free entries in $a_{j+1}$, call them $\gamma<\delta$, and $a_j$ is obtained by releasing a different joined pair.

  We are now able to construct our involution.
  To our alternating path $c$ as defined above, we assign the alternating path $c'$ that agrees with $c$ on $(\sigma,\ldots,u(a_j))$ but redefines $a_{j+1}'$ to be the cell obtained from $u(a_j)$ by releasing $\gamma$ and $\delta$ in the opposite order from $a_{j+1}$.
  We then inductively define $a_i',u(a_i'),a_{i+1}'$ by using swaps starting with $a_{j+1}$.
  The requirement that $a_i'$ is matched to $u(a_i')$ determines the swap that will take place, and ensures that all of the joined elements in our original cell $\tau$ remain joined throughout this process.
  Further, because we only use swaps, and each swap decreases the number of inversions in the underlying word for $\tau$ by one, it follows that $c'$ has length either $t-1$ or $t+1$ and is also an alternating path from $\sigma$ to $\tau$.
  Because both $c$ and $c'$ terminate in a sequence of swaps following $u(a_j)$, it follows that $c$ and $c'$ are bijectively mapped to each other, and we have established an involution.

  We next need to show that $[\tau:\sigma]=0$ for all such critical cells. 
  From~\eqref{eqn:weight}, we have that
  \[
    [\tau:\sigma]=\sum_cw(c)
  \]
  with $c$ ranging over all alternating paths from $\sigma$ to $\tau$ and
  \[
    w(c):=(-1)^t[a_1:\sigma][\tau:u(a_t)]\prod_{i=1}^t[a_i:u(a_i)]\prod_{i=1}^{t-1}[a_{i+1}:u(a_i)]  \, .
  \]
  From~\eqref{eqn:pairincidence} it follows that for any swap $(a_i,u(a_i),a_{i+1})$ we have that $[a_i:u(a_i)]\cdot[a_{i+1}:u(a_i)]=-1$.
  Thus, if we define $A:=[a_1:\sigma]\prod_{i=1}^j[a_i:u(a_i)]\cdot \prod_{i=1}^{j-1}[a_{i+1}:u(a_i)]$, we have
  \begin{align*}
    w(c)\cdot w(c')& =(-1)^t(-1)^{t\pm 1}\cdot A\cdot A \cdot (-1)^{\# \text{ swaps in }c}\cdot (-1)^{\# \text{ swaps in }c'} \cdot [a_{j+1}:u(a_j)]\cdot [a_{j+1}':u(a_j)] \\
    & = -1
  \end{align*}
  where the product of the final two terms is $-1$ due to~\eqref{eqn:pairincidence}.
  Hence, our involution is sign-reversing, and our matching is optimal.

\end{proof}

\begin{example}
  The alternating path given by
  \begin{align*}
    c=[&(7(63)9(81)5(42),\, 7(63)9(81)542,\, 7(63)9(81)(54)2,\\
       &7(63)918(54)2,\, 7(63)(91)8(54)2,\, 7(63)198(54)2,\\
       &7(63)1(98)(54)2,\, 7(63)189(54)2]
  \end{align*}
  ends with the final five cells involved in swaps.
  Thus, in the third cell of $c$, we release $(81)$ in the opposite order to obtain $c'$ below, as described in the proof of optimality.
  Note that we complete the fourth and later cells listed in $c'$ by requiring that all subsequent cells be obtained as swaps.
  \begin{align*}
    c'=[&(7(63)9(81)5(42),\, 7(63)9(81)542,\, 7(63)9(81)(54)2,\\
       &7(63)981(54)2,\, 7(63)(98)1(54)2,\, 7(63)891(54)2,\\
       &7(63)8(91)(54)2,\, 7(63)819(54)2,\, 7(63)(81)9(54)2,\\
       &7(63)189(54)2]
  \end{align*}

\end{example}

\section{Examples and Special Cases}\label{sec:examples}


\subsection{$\bi = (r)$ or $(r,s)$}

When $\bi = (r)$, the poset is a single chain, in which case our complex $\Hom{(r)}$ is a single point.
When $\bi = (r,s)$, it is a straightforward exercise to show that the corresponding product of chains is a rectangular grid that collapses via folds to a single chain.
Thus, $\Hom{(r,s)}$ is contractible by Theorem~\ref{posetfold}.


\subsection{$\bi = (r,s,t)$}

When our poset is a product of three chains, the structure of $\Hom(\bi)$ becomes more interesting.

\begin{theorem}
  The number of critical $k$-cells in $\Hom(r,s,t)$ is $\binom{r}{k} \cdot \binom{s}{k} \cdot \binom{t}{k}$ where $k \leq r \leq s \leq t$.
\end{theorem}

\begin{proof}
  The complex $\Hom(r,s,t)$ corresponds to words from the multiset $\{1^r, 2^s, 3^t\}$.
  The only possible critical configuration for such a word has a transposition of the form $3(21)$ with its preceding ascending free sequence listing its free 1's, followed by its free 2's, and ending with its free 3's.

  Consider the list $1_1, 1_2, \dots, 1_r$.
  Select $k$ of these $1$'s, say $1_{j_1}, 1_{j_2}, \dots, 1_{j_k}$.
  Recall that there are $\binom{r}{k}$ such ways to make these selections.
  We will distribute the entries $1_1, 1_2, \dots, 1_r$ into a critical cell $\sigma$ of dimension $k$.
  We put the entry $1_{j_1}$ into the first transposition of $\sigma$ while the entries $1_1, 1_2, \dots, 1_{{j_1}-1}$ go into the free sequence preceding this first transposition. 
  For $1 < n < k$, we put $1_{j_n}$ into the $n$-th transposition while $1_{{j_n}+1}, \dots, 1_{{j_{n+1}}-1}$ all go into the free sequence between the $n$-th and $(n+1)$-th transpositions.
  Finally, we put $1_{j_k}$ into the final transposition, and  $1_{{j_k}+1}, \dots, 1_r$ all go into the terminal free sequence.

  We then select $k$ of the $s$ 2's in one of $\binom{s}{k}$ ways, select $k$ of the $t$ 3's in one of $\binom{t}{k}$ ways, and perform a similar distribution process.
  The selected 2's should be placed in the transpositions with the selected 1's, and the selected 3's become the free entries just before their corresponding transpositions.
  It is straightforward to reverse this procedure and convert a critical cell in $\Hom(r,s,t)$ back into a selection of 1's, 2's, and 3's.

  Since the selections for each letter are independent of the selections of the other letters, we have $\binom{r}{k} \cdot \binom{s}{k} \cdot \binom{t}{k}$ total ways to make selections and create a critical cell.
\end{proof}

\begin{example}
  In $\Hom(4,4,4)$, the critical cell $233(21)123(21)13$ corresponds to the following selections

  \begin{center}
    \fbox{$1_1$}, $1_2$, \fbox{$1_3$}, $1_4$ \qquad $2_1$, \fbox{$2_2$}, $2_3$, \fbox{$2_4$} \qquad $3_1$, \fbox{$3_2$}, \fbox{$3_3$}, $3_4$
  \end{center}

\end{example}


\subsection{$\bi=(1,1,\ldots,1)$}

We first recall that $C_1^n\cong B_n$ where $B_n$ is the Boolean algebra on $n$ symbols.
Thus, $\Hom(1,1,\ldots,1)=\Hom(B_n)$.
As mentioned earlier, this complex has been studied using discrete Morse theory by Severs and White~\cite{severswhite} in a different context; our results mirror theirs.

\begin{theorem}\label{thm:booleaneuler}
  The unreduced Euler characteristic of $\Hom(B_n)$ is
  \begin{equation}
    \chi_n = \sum^{\lfloor n/2 \rfloor}_{k=0} (-1)^k \frac{n!}{2^k}\binom{n-k}{k}. \label{eqn:homEuler}
  \end{equation}
\end{theorem}

\begin{proof}
  Fix $n \geq 1$.
  We now count the number of $k$-cells in $\Hom(B_n)$ for $0 \leq k \leq \lfloor n/2 \rfloor$.
  Since the dimension of a cell in $\Hom(B_n)$ equals the number of transpositions in its parenthesized notation, a $k$-cell will have transpositions in $k$ different locations.
  If we consider a free element as a block of size 1 and a transposition as a block of size 2, then we see that the relative order of free elements and transpositions for a given $k$-cell corresponds to a composition of $n$ using parts of size 1 and 2.
  We know there are $\binom{n-k}{k}$ such compositions.
  (There will be $n-k$ blocks, and we choose $k$ of them to be of size 2.)
  Now, for each such block arrangement, there are $\frac{n!}{2^k}$ ways to order the elements into the blocks where the order in a block of size 2 is irrelevant.
  In total, there are $\frac{n!}{2^k}\binom{n-k}{k}$ cells of size $k$ in $\Hom(B_n)$.
\end{proof}

\begin{theorem}\label{thm:eulerrecur}
  Equation~\eqref{eqn:homEuler} satisfies the recursion $\chi_n = n \cdot \chi_{n-1} - \binom{n}{2} \cdot \chi_{n-2}$ with initial conditions $\chi_1 = \chi_2 = 1$.
\end{theorem}

\begin{proof}
  It is straightforward to verify the initial conditions.
  Let $X_n = \dfrac{\chi_n}{n!}$ for arbitrary $n \geq 1$, and 
  we prove instead that $X_n = X_{n-1} - \frac{1}{2} X_{n-2}$.
  Multiplying by $n!$ gives the result.

  First, suppose that $n = 2j$ for some integer $j$.
  From formula \ref{eqn:homEuler}, we know that
  \begin{align*}
    X_{n-1} &= \sum^{j-1}_{k=0} \left( \frac{-1}{2} \right)^k \binom{n-1-k}{k} \\
            &= \left( \frac{-1}{2} \right)^0 \binom{n-1-0}{0} \quad  + \quad \sum^{j-1}_{k=1} \left( \frac{-1}{2} \right)^k \binom{n-1-k}{k}. \\
  \end{align*}

  We also have, via reindexing, 
  \begin{align*}
    -\frac{1}{2} X_{n-2} &= \sum^{j-1}_{k=0} \left( \frac{-1}{2} \right)^{k+1} \binom{n-2-k}{k} \\
                         &= \sum^{j}_{k=1} \left( \frac{-1}{2} \right)^k \binom{n-1-k}{k-1} \\
                         &= \sum^{j-1}_{k=1} \left( \frac{-1}{2} \right)^k \binom{n-1-k}{k-1} \quad + \quad \left( \frac{-1}{2} \right)^j \binom{n-1-j}{j-1}. \\
  \end{align*}

  Using the fact that $\binom{n-1-k}{k-1} + \binom{n-1-k}{k} = \binom{n-k}{k},$ we compute
  $$X_{n-1} \ - \ \frac{1}{2}X_{n-2} \ \ = \ \ \left( \frac{-1}{2} \right)^0 \binom{n-1-0}{0} \ \  + \ \ \sum^{j-1}_{k=1} \left( \frac{-1}{2} \right)^k \binom{n-k}{k} \ \ + \ \ \left( \frac{-1}{2} \right)^j \binom{n-1-j}{j-1}.$$

  Recognizing that $\binom{n-1-0}{0} = \binom{n-0}{0}$ and that $j = n/2$ implies $\binom{n-1-j}{j-1} = \binom{n-j}{j}$, we obtain
  $$X_{n-1} \ \ - \ \ \frac{1}{2}X_{n-2} \ \ = \ \ \sum^{j}_{k=0} \left( \frac{-1}{2} \right)^k \binom{n-k}{k} \ \ = \ \ X_n.$$

  A similar argument holds for the case when $n = 2j+1$.

  \commentout{
    \hspace{-0.275in} \textbf{\underline{Case 2}:}
    Suppose that $n = 2j+1$ for some integer $j$.
    Formula \ref{eqn:homEuler} gives $$X_{n-1} = \sum^{j}_{k=0} \left( \frac{-1}{2} \right)^k \binom{n-1-k}{k}.$$
    We also have $$-\frac{1}{2}X_{n-2} = \sum^{j-1}_{k=0} \left( \frac{-1}{2} \right)^{k+1} \binom{n-2-k}{k} = \sum^{j}_{k=1} (-1)^{k} \frac{1}{2^{k}}\binom{n-1-k}{k-1},$$ where the second equality is obtained by reindexing.
    Therefore, $$X_{n-1} - \frac{1}{2}X_{n-2} = 1 + \sum^{j}_{k=1} \left( \frac{-1}{2} \right)^k \binom{n-1-k}{k} + \sum^{j}_{k=1} (-1)^{k} \frac{1}{2^{k}}\binom{n-1-k}{k-1}.$$
    Since $\binom{n-1-k}{k-1} + \binom{n-1-k}{k} = \binom{n-k}{k}$, we have 
    $$X_{n-1} - \frac{1}{2}X_{n-2} = 1 + \sum^{j}_{k=1} \left( \frac{-1}{2} \right)^k \binom{n-k}{k}.$$
    Observe that $(-1)^0 \frac{1}{2^0}\binom{n-0}{0} = 1$.
    Thus, $$X_{n-1} - \frac{1}{2}X_{n-2} = \sum^{j}_{k=0} \left( \frac{-1}{2} \right)^k \binom{n-k}{k} = X_n.$$
    \hspace{-0.275in} Thus, the result holds in both cases.
  }

\end{proof}

\begin{corollary}
  If we write $n = 4q + r$ where $0 \leq r \leq 3$ and $q$ is a non-negative integer, then
  $$\chi_n = \left\{ \begin{array}{ll}
                       (-1/4)^q \cdot n! & r \in \{0,1\} \\
                       (1/2) (-1/4)^q \cdot n! & r = 2 \\
                       0 & r = 3 \\
                     \end{array} \right.$$
                 \end{corollary}

                 \begin{proof}
                   It is straightforward to verify that $\chi_1 = \chi_2 = 1$, $\chi_3 = 0$, and $\chi_4 = -6$.
                   Let $X_n = \chi_n / n!$ as before, and recall that $X_n = X_{n-1} - \frac{1}{2}X_{n-2}$ for $n \geq 3$.
                   Then, we have $X_1 = 1$, $X_2 = 1/2$, $X_3 = 0$, and $X_4 = -1/4$. 
                   Now, for $n \geq 5$, we see that

                   $$\begin{array}{ll}
                       X_n &= X_{n-1} - \frac{1}{2}X_{n-2} \\
                           &= (X_{n-2} - \frac{1}{2}X_{n-3}) - \frac{1}{2}(X_{n-3} - \frac{1}{2}X_{n-4}) \\
                           &= X_{n-2} - X_{n-3} + \frac{1}{4}X_{n-4} \\
                           &= (X_{n-3} - \frac{1}{2}X_{n-4}) - X_{n-3} + \frac{1}{4}X_{n-4} \\
                           &= -\frac{1}{4}X_{n-4}.\\
                     \end{array}$$

                     Hence, $X_n = (-1/4)^q X_r$ for $n = 4q + r$ as in the statement.
                     Technically, $X_0$ is undefined, but we set it equal to $1$ to ensure that $X_4 = (-1/4)^1 X_0 = -1/4$ compatibly.
                     Multiplying by $n!$ proves the claim.
                   \end{proof}


                   \bibliographystyle{plain}
                   \bibliography{Braun}

                 \end{document}